\documentclass[11pt,english,a4paper]{amsart}

\usepackage[utf8]{inputenc}
\usepackage[abbrev,backrefs,nobysame]{amsrefs}
\usepackage{amsmath,graphicx}
\usepackage{amssymb}
\usepackage{blkarray}
\usepackage{amsthm}
\usepackage{enumerate}
\usepackage{tikz-cd}
\usepackage{mathrsfs}
\usepackage{dsfont}
\usepackage{pst-all}
\usepackage{amsfonts}
\usepackage{amsxtra}

\theoremstyle{definition}
\newtheorem{Def}{Definition}[section]
\newtheorem{Thm}[Def]{Theorem}
\newtheorem{Prop}[Def]{Proposition}

\newtheorem{Lem}[Def]{Lemma}

\newtheorem{Ques}[Def]{Question}
\newtheorem{Clm}[Def]{Claim}

\newcommand{\adhe}[1]{\overline{#1}}

\newcommand{\accolade}[2]{\left\{\begin{array}{#1} #2 \end{array} \right.}
\newcommand{\pdtsca}[1]{\langle #1 \rangle}
\newcommand{\abs}[1]{\left| #1 \right|}

\title{
	Linear dynamics of an operator associated to the Collatz map
}
\author{
	Vincent B\'ehani
}
\address{
	Univ. Lille, CNRS, UMR 8524 - Laboratoire Paul Painlev\'e, F-59000 Lille, France
}
\email{
	vincent.behani@univ-lille.fr
}
\date{\today}

\begin{document}
	
	\thanks{
		This work was supported in part by the project FRONT of the French National Research Agency (grant ANR-17-CE40-0021) and by the Labex CEMPI (ANR-11-LABX-0007-01). 
	}

	\thanks{
		I am grateful to Sophie Grivaux for our discussions and her helpful proofreadings.
	}


	
	\begin{abstract}
		In this paper, we study the dynamics of an operator $\mathcal T$ naturally associated to the so-called \textit{Collatz map}, which maps an integer $n \geq 0$ to $n / 2$ if $n$ is even and $3n + 1$ if $n$ is odd.
		This operator $\mathcal T$ is defined on certain weighted Bergman spaces $\mathcal B ^ 2 _ \omega$ of analytic functions on the unit disk.
		Building on previous work of Neklyudov, we show that $\mathcal T$ is hypercyclic on $\mathcal B ^ 2 _ \omega$, independently of whether the Collatz Conjecture holds true or not.
		Under some assumptions on the weight $\omega$, we show that $\mathcal T$ is actually ergodic with respect to a Gaussian measure with full support, and thus frequently hypercyclic and chaotic.
	\end{abstract}

	\maketitle
		
	\section{Introduction}
	
		Our aim in this paper is to investigate from the point of view of linear dynamics the properties of an operator naturally associated to the so-called Collatz map.
		This map $T _ 0 \colon \mathbf Z _ + \to \mathbf Z _ +$ is defined as
		\[
			T _ 0(n) = 
			\accolade{l l}{
				n / 2 & \text{if $n$ is even;} \\
				3n + 1 & \text{if $n$ is odd.}
			}
		\]
		The dynamics of this simple-looking map are still very mysterious, and the famous Collatz Conjecture (called also the $3n + 1$-Conjecture, or the Syracuse Conjecture) states that the orbit of any integer $k \geq 1$ under the action of $T _ 0$ eventually reaches the point 1.
		
		We will use in this paper the following modified Collatz map $T\colon \mathbf Z _ +\to \mathbf Z _ +$, defined as
		\[
			T(n) = 
			\accolade{l l}{
				n / 2 & \text{if $n$ is even;} \\
				(3n + 1) / 2 & \text{if $n$ is odd}
			}
		\]
		whose dynamics are the same as those of $T _ 0$ as far as the Collatz Conjecture is concerned.
		
		We will here give some of the etablished results about the Collatz Conjecture.
		We refer the reader to the survey \cite{lag10} written by Lagarias for an overview of this problem.
		
		The conjecture had been empirically verified for every $k \leq 20 \times 10 ^ 9$ and is now known to be true for every $k \leq 20 \times 2 ^ {58} \approx 5.764 \times 10 ^ {18}$ (\cite{oli10}).
		
		Since $T(1) = 2$ and $T(2) = 1$, $(1, 2)$ is called the trivial cycle of $T$ and Eliahou showed in \cite{eli93} that a non-trivial cycle of $T$ must have a length greater than 10 439 860 591.
		
		A possible approach to the Collatz Conjecture is to study the proportion of positive integers $k$ having a $T$-orbit that reaches 1.
		Krasikov and Lagarias proved for instance in \cite{kra02} that if $X$ is sufficiently large, the number of integers $1 \leq k \leq X$ whose $T$-orbit reaches 1 is at least $X ^ {0.84}$.
		
		It is also of interest to investigate the time that it takes to a $T$-orbit to reach 1.
		Applegate and Lagarias proved in \cite{app02} that infinitely many positive integers $k$ have a $T$-orbit that reaches 1 in at least $6.143\log(k)$ steps.
		
		Moreover, Tao showed in \cite{tao22} that "almost all" $T$-orbits attain almost bounded values.
		That is to say, if we fix a function $f \colon \mathbf N \to \mathbf R$ such that $f(k) \to + \infty$ as $k \to + \infty$, the infinimum of the $T$-orbit of $k$ is less than $f(k)$ for almost every integer $k \geq 1$, in the sense of logarithmic density.
		
		It is possible to associate to $T$ in a natural way bounded operators on some Hilbert spaces of analytic functions on the unit disk $\mathbf D$, and to link the dynamics of $T$ to those of these operators.
		An approach of this kind was first proposed by Berg and Meinardus in \cite{ber94}. Consider the operator $\mathcal F$ defined on the space 
		$Hol(\mathbf D)$
		of holomorphic functions on the unit disk $\mathbf D$
		in the following way:
		for every $f \in Hol(\mathbf D)$
		with 
		$
			f(z) = 
			\sum _ {n = 0} ^ \infty 
				c _ n z ^ n,
		$
		\[
			\mathcal Ff(z) =
			\sum _ {n = 0} ^ \infty 
				c _ {T(n)} z ^ n,
			\quad
			z \in \mathbf D.
		\]
		Berg and Meinardus expressed in \cite{ber94} the Collatz Conjecture in terms of functional equations and obtained that it is equivalent to the fact that 1 is an eigenvalue of $\mathcal F$ of multiplicity 2. 
		
		Then Neklyudov took in \cite{nek21} a different perpective, and considered the adjoint $\mathcal T$, for the Hardy space $\mathcal H ^ 2(\mathbf D)$, of the operator $\mathcal F$:
		for every $f \in \mathcal H ^ 2(\mathbf D)$ with 
		$
			f(z) = 
			\sum _ {n = 0} ^ \infty
				c _ n z ^ n,
		$
		\[
			\mathcal T f(z) =
			\sum _ {n = 0} ^ \infty
				c _ n z ^ {T(n)},
			\quad
			z \in \mathbf D.
		\]
		This operator $\mathcal T$ is considered in \cite{nek21} as acting on the Bergman space
		\[
			\mathcal B ^ 2 =
			\left\{
				f \colon z \mapsto 
				\sum _ {n = 0} ^ \infty
					c _ n z ^ n
				\in Hol(\mathbf D) ;
				\|f \| ^ 2 =
				\sum _ {n = 0} ^ \infty
					\frac{
						\pi \abs{c _ n} ^ 2
					}{
						n + 1
					}
				< + \infty
			\right\}
		\]
		on $\mathbf D$.
		Since $T(0) = 0$, $T(1) = 2$ and $T(2) = 1$, the vector space $\mathcal H _ 0 = \text{span}[1, z, z ^ 2]$ is invariant under the action of $\mathcal T$, and $\mathcal T$ induces a bounded operator on the quotient space $\mathcal X = \mathcal B ^ 2 / \mathcal H _ 0$, which we still denote by $\mathcal T$.
		Considering this quotient space allows to avoid the trivial cycle $(1, 2)$ and the fixed point 0 of the Collatz map, and simplifies the study of the dynamics of $\mathcal T$.
		Neklyudov proved in \cite{nek21} several results pertaining to the behaviour of the iterates of $\mathcal T$, in connection with the Collatz Conjecture.
		For instance, he showed that if $T$ has no non-trivial cycle, then $\mathcal T$ is hypercyclic, i.e. admits a vector with dense orbit \cite[Theorem 2.2]{nek21}.
		He also undertook a study of the periodic points of $\mathcal T$ and showed that each hypothetical cycle of $T$ and each hypothetical diverging $T$-orbit can be associated to a fixed point of $\mathcal T$.
		He also constructed fixed points of $\mathcal T$ independently of any cycle or diverging orbit of $T$.

		Neklyudov asked in \cite{nek21} about the existence of a criterion allowing to distinguish these fixed points and also whether $\mathcal T$ is chaotic (i.e. hypercyclic and admitting a dense set of periodic points).
		
		In this work, we continue this study of operators on Hilbert spaces of analytic functions associated to the Collatz map and give an affirmative answer to this question of Neklyudov.
		We consider the operator $\mathcal T$ above as acting on weighted Bergman spaces of holomorphic functions on $\mathbf D$ of the form
		\[
			\mathcal B ^ 2 _ \omega = 
			\left\{
				f \colon z \mapsto
				\sum _ {n = 0} ^ \infty
					c _ n z ^ n
				\in Hol(\mathbf D) ;
				\|f \| ^ 2 _ \omega =
				\sum _ {n = 0} ^ \infty
					\frac{
						\abs{c _ n} ^ 2
					}{
						\omega(n)
					}
				< + \infty
			\right\}
		\]
		where $\omega \colon \mathbf Z _ + \to (0, + \infty)$ is a positive weight.
		Just as above, it will be more natural to study the action of $\mathcal T$ on the quotient space $\mathcal X _ \omega = \mathcal B ^ 2 _ \omega / \mathcal H _ 0$.
		Let $\omega _ 0(n) = (n + 1) / \pi$ for every $n \geq 0$.
		Then $\mathcal B ^ 2 _ {\omega _ 0}$ is the classical Bergman space $\mathcal B ^ 2$.
		We write $\mathcal X _ {\omega _ 0} = \mathcal X$.
		
		Whenever the weight $\omega$ is such that $\mathcal T$ defines a bounded operator on $\mathcal X _ \omega$, it is of interest to study the dynamics of $\mathcal T$.
		The main properties that we will consider are those of hypercyclicity (existence of a dense orbit, or, equivalently, topological transitivity),  chaos, frequent hypercyclicity, ergodicity with respect to an invariant measure with full support...
		We give a brief overview of these notions in Section 3 of the present paper, and refer the reader to one of the works \cite{bay09} and \cite{gro11} for more on linear dynamics.
		
		Our first main result states that under some mild assumptions on the weight $\omega$, the operator $\mathcal T$ is hypercyclic. 
		
		\begin{Thm}\label{intro1}
			If 
			$
				\omega
			$ 
			is bounded from below
			and if
			$
				\omega(k2 ^ n) \to + \infty
			$
			as 
			$
				n \to + \infty
			$
			for every $k \geq 3$,
			then $\mathcal T$ is hypercyclic.
		\end{Thm}
		Theorem \ref{intro1} applies in particular to the weight $\omega _ 0$, i.e. to the case where $\mathcal T$ acts on the Bergman space, and shows that the result \cite[Theorem 2.2]{nek21} of Neklyudov is in fact not conditional to the Collatz map having no non-trivial cycle.
		
		Our second main result shows, under slightly stronger assumptions on $\omega$, that $\mathcal T$ acting on $\mathcal X _ \omega$ enjoys some strong dynamical properties:
		
		\begin{Thm}\label{intro2}
			If 
			$
				\omega
			$
			is bounded from below
			and if 
			$
			\sum _ {n = 0} ^ \infty
				1 / \omega(k2 ^ n)
				< + \infty
			$
			for every $k \geq 3$, then $\mathcal T$ is chaotic, fequently hypercyclic and ergodic with respect to a Gaussian invariant measure with full support.
		\end{Thm}
	
		In particular, $\mathcal T$ acting on the Bergman space $\mathcal X$ is chaotic.
		This provides an affirmative answer to Neklyudov's question.
		
		The paper is organized as follows. 
		We first provide in Section 2 conditions on the weight $\omega$ ensuring that $\mathcal T$ acts as a bounded operator on $\mathcal X _ \omega$ (Proposition \ref{norm}); building on a description given in \cite{nek21} of some eigenvectors of $\mathcal T$, we show (still under some mild conditions on $\omega$) that these eigenvectors span a dense subspace of $\mathcal X _ \omega$ (Theorem \ref{densityeigenvectors}) and that the adjoint $\mathcal T ^ *$ of $\mathcal T$ on $\mathcal X _ \omega$ has empty point spectrum (Theorem \ref{noeigenvalues}).
		After recalling in Section 3 some background in linear dynamics, we prove Theorems \ref{intro1} and \ref{intro2} above.
		The proofs of these two results rely heavily on the properties of the eigenvectorfields of $\mathcal T$ presented in Section 2, and in particular on Theorem \ref{densityeigenvectors}.
		Section 4 collects some additional results as well as some open questions.

	\section{Boundedness of $\mathcal T$ and properties of its eigenvectors}

	Let $\omega$ be a positive weight on $\mathbf Z _ +$.
	We first consider the operator $\mathcal T$ acting on the space
	$
		\mathbf C[\xi] / \text{span}[1, \xi, \xi ^ 2]
	$, defined by
	\[
		\mathcal T 
		\sum _ {n = 3} ^ d
			c _ n \xi ^ n
		=
		\sum _ {3 \leq n \leq d, T(n) \geq 3}
			c _ n \xi ^ {T(n)}
		\quad
		\text{for every }
		\sum _ {n = 3} ^ d
			c _ n \xi ^ n
		\in \mathbf C[\xi] / \text{span}[1, \xi, \xi ^ 2],
	\]
	and give conditions on $\omega$ implying that $\mathcal T$ can be extended to a bounded operator on $\mathcal X _ \omega$.
		
	\begin{Prop}\label{norm}		
		The operator $\mathcal T$, defined for every 
		$
			f \in \mathcal X _ \omega
		$
		with 
		$
			f(z) =
			\sum _ {n = 3} ^ \infty
				c _ n z ^ n
		$
		by
		\[
			\mathcal T 
			\sum _ {n = 3} 
				c _ n z ^ n 
			=
			\sum _ {T(j) \geq 3}
				c _ j z ^ {T(j)},
			\quad
			z \in \mathbf D,	
		\] 
		is a bounded operator acting on $\mathcal X _ \omega$ if and only if the three sequences 
		$
			(\omega(6m) / \omega(3m)) _ {m \geq 1}
		$,
		$
			(\omega(6m + 2) / \omega(3m + 1)) _ {m \geq 1}
		$
		and
		$
			((\omega(6m + 4) + \omega(2m + 1)) / \omega(3m + 2)) _ {m \geq 1}
		$
		are bounded.
		In this case
		\[
			\|\mathcal T \| _ \omega ^ 2 =
				\max\left\{
					\sup _ {m \geq 1}
						\frac{\omega(6m)}{\omega(3m)}, 
					\sup _ {m \geq 1}
						\frac{\omega(6m + 2)}{\omega(3m + 1)},
					\sup _ {m \geq 1}
						\frac{
							\omega(6m + 4) + \omega(2m + 1)
						}{
							\omega(3m + 2)
					}
			\right\}.
		\]
		Moreover for every $n \geq 0$
		\[
			\|\mathcal T ^ n \| _ \omega ^ 2 =
			\sup _ {k \geq 3}
				\sum _ {T ^ n(j) = k}
					\frac{
						\omega(j)
					}{
						\omega(k)
					}.
		\]
		Taking $\omega = \omega _ 0$ we obtain in particular that $\mathcal T$ is bounded on $\mathcal X$ and that $\|\mathcal T\| ^ 2 = 8 / 3$.
	\end{Prop}
	\begin{proof}
		Suppose that the three sequences above are bounded, which means that the maximum
		\[
			\max\left\{
				\sup _ {m \geq 1}
					\frac{
						\omega(6m)
					}{
						\omega(3m)
					},
				\sup _ {m \geq 1}
					\frac{
						\omega(6m + 2)
					}{
						\omega(3m + 1)
					},
			\sup _ {m \geq 1}
				\frac{
					\omega(6m + 4) + \omega(2m + 1)
				}{
					\omega(3m + 2)
				}
			\right\}
			=
			\sup _ {k \geq 3}
				\sum _ {T(j) = k} 
					\frac{
						\omega(j)
					}{
						\omega(k)
					}
		\]
		is finite.
		Let $n \geq 0$ and
		$
			f \in \mathcal X _ \omega
		$
		with 
		$
			f(z) = 
			\sum _ {k = 3} ^ \infty
				c _ k z ^ k
		$.
		We have
		\begin{align*}
			\mathcal T ^ n f(z)
			=
			\sum _ {T ^ n (j) \geq 3}
				c _ j z ^ {T ^ n(j)}
			=
			\sum _ {k = 3} ^ \infty
				\sum _ {T ^ n(j) = k}
					c _ j
				z ^ k,
			\quad 
			z \in \mathbf D,
		\end{align*}
		and
		\[
			\|\mathcal T ^ n f \| _ \omega ^ 2 =
			\sum _ {k = 3} ^ \infty
				\frac 1 {\omega(k)}
				\abs{
					\sum _ {T ^ n(j) = k}
						c _ j
				} ^ 2
			=
			\sum _ {k = 3} ^ \infty
				\frac 1 {\omega(k)}
				\abs{
					\sum _ {T ^ n(j) = k}
						\sqrt{\omega(j)}
						\frac{
							c _ j
						}{
							\sqrt{\omega(j)}
						}
				} ^ 2.
		\]
		By the Cauchy-Schwarz's inequality
		\[
			\|\mathcal T ^ n f\| ^ 2 _ \omega
			\leq
			\sum _ {k = 3} ^ \infty
				\sum _ {T ^ n(j) = k}
					\omega(j)
				\sum _ {T ^ n(j) = k}
					\frac{
						\abs{c _ j} ^ 2
					}{
						\omega(j)
					}
			\leq
			\left(
				\sup _ {k \geq 3}
					\sum _ {T ^ n(j) = k}
						\frac{
							\omega(j)
						}{
							\omega(k)
						}
			\right)
			\|f \| ^ 2 _ \omega,
		\]
		which gives that $\mathcal T$ is bounded if the sequence 
		$
			(\sum _ {T(j) = k}
				\omega(j) / \omega(k)
			) _ {k \geq 3}
		$
		is bounded and that
		$
			\|\mathcal T ^ n \| ^ 2 _ \omega
			\leq
			\sup _ {k \geq 3}
				\sum _ {T ^ n(j) = k}
					\omega(j) / \omega(k).
		$
		
		Conversely, let $n \geq 0$ and suppose that $\mathcal T$ is a bounded operator on $\mathcal X _ \omega$.
		Let $(k _ p) _ {p \geq 1}$ such that 
		$
			\sum _ {T ^ n(j) = k _ p}
					\omega(j)
				/
				\omega(k _ p)
			\to 
			\sup _ {k \geq 3}
				\sum _ {T ^ n(j) = k}
						\omega(j)
					/
					\omega(k)
		$
		as $p \to + \infty$.
		Consider the function
		$
			f _ p \in \mathcal X _ \omega
		$
		with 
		$
			f(z) =
			\sum _ {T ^ n(j) = k _ p}
				\omega(j)z ^ j
		$
		satisfying
		$
			\|f _ p\| ^ 2 _ \omega 
			=
			\sum _ {T ^ n(j) = k _ p}
					\abs{\omega(j)} ^ 2
				/
					\omega(j)
			=
			\sum _ {T ^ n(j) = k _ p}
				\omega(j)
		$
		and
		\[
			\|\mathcal T ^ n f _ p\| _ \omega ^ 2 
			=
			\frac 1 {\omega(k _ p)}
			\abs{
				\sum _ {T ^ n(j) = k _ p}
				\omega(j)
			} ^ 2
			=
			\left(
				\sum _ {T ^ n(j) = k _ p}
					\frac{\omega(j)}{\omega(k _ p)}
			\right)
			\|f _ p\| _ \omega ^ 2.
		\]
		So we have 
		$
				\|\mathcal T ^ n f _ p \| _ \omega ^ 2
				/
				\|f _ p\| _ \omega ^ 2
			\to 
			\sup _ {k \geq 3}
				\sum _ {T ^ n(j) = k}
					\omega(j)
				/
				\omega(k)
		$
		as $p \to + \infty$
		and it follows that
		$
			\sup _ {k \geq 3}
				\sum _ {T ^ n(j) = k}
					\omega(j) / \omega(k)
			\leq 
			\|\mathcal T ^ n \| ^ 2 _ \omega
			< + \infty
		$
		for every $n \geq 0$,
		which concludes the proof.
	\end{proof}

	For instance, it follows from Proposition \ref{norm} that if $\omega$ has polynomial growth, then $\mathcal T$ is a bounded operator on $\mathcal X _ \omega$.\smallskip
	
	Eigenvectors play a very important role in linear dynamics. 
	We will thus be interested in the eigenvalues and the eigenvectors of $\mathcal T$, and will provide some condition on the weight $\omega$ ensuring that the eigenvectors of $\mathcal T$ span a dense subspace of $\mathcal X _ \omega$.
	
	\begin{Def}
		For every $\mu \in \mathbf C$ and every $m \geq 1$, consider the analytic functions on $\mathbf D$
		\[
			h _ m(\mu, \cdot) \colon z \mapsto 
			\sum _ {n = 0} ^ \infty
				\mu ^ n 
				\left(
					z ^ {(6m + 4) 2 ^ n} - z ^ {(2m + 1)2 ^ n}
				\right)
			\quad
			\text{and}
			\quad
			h _ 0(\mu, \cdot) \colon z \mapsto 
			\sum _ {n = 0} ^ \infty
				\mu ^ n z ^ {2 ^ {n + 2}}.
		\]
	\end{Def}

	These functions has been introduced by Neklyudov in \cite[Theorem 2.3]{nek21} and are the only eigenvectors that will be needed in order to obtain results concerning the spanning of dense subspaces.

	\begin{Prop}\label{eigenvectors}
		For every $m \geq 0$ and every $\mu \in \mathbf C$, the vector $h _ m(\mu, \cdot)$
		belongs to $\mathcal X _ \omega$ if and only if the series 
		$
			\sum _ {n \geq 0}
				\abs{\mu} ^ {2n} / \omega((2m + 1) 2 ^ n)
		$ 
		and
		$
			\sum _ {n \geq 0}
				\abs{\mu} ^ {2n} / \omega((6m + 4) 2 ^ n)
		$ 
		converge.
		In this case 
		$
			\mathcal Th _ m(\mu, \cdot) = \mu h _ m (\mu, \cdot)
		$. 
		
		In particular if $\omega$ is bounded from below then $h _ m(\mu, \cdot)$ belongs to $\mathcal X _ \omega$ for every $m \geq 0$ and every $\mu \in \mathbf D$.		
		Moreover if $\omega = \omega _ 0$ then $h _ m(\mu, \cdot)$ belongs to $\mathcal X$ for every $m \geq 0$ and every $\mu \in \mathbf C$ such that $\abs\mu < \sqrt 2$.
	\end{Prop}
	\begin{proof}
		Let $\mu \in \mathbf C$ and $m \geq 1$.
		The function 
		$
			h _ m(\mu, \cdot)
		$
		belongs to
		$
			\mathcal X _ \omega
		$
		if and only if 
		$
			\sum _ {n = 0} ^ \infty
				\abs\mu ^ {2n} / \omega((6m + 4)2 ^ n)
			+
			\sum _ {n = 0} ^ \infty
				\abs\mu ^ {2n} / \omega((2m + 1)2 ^ n)
			< + \infty
		$.
		And
		$
			h _ 0(\mu, \cdot)
		$
		belongs to 
		$
			\mathcal X _ \omega
		$
		if and only if the series
		$
			\sum _ {n = 0} ^ \infty 
				\abs\mu ^ {2n} / \omega(2 ^ {n + 2})
			< + \infty
		$.
		
		Moreover for every $m \geq 1$ and every $z \in \mathbf D$ we have
		\begin{align*}
			\mathcal T h _ m(\mu, \cdot)(z)
			& =
			\sum _ {n = 0} ^ \infty
				\mu ^ n 
				\left(z ^ {T((6m + 4) 2 ^ n)} - z ^ {T((2m + 1)2 ^ n)}\right) \\
			& =
			\sum _ {n = 1} ^ \infty
				\mu ^ n
				\left(z ^ {(6m + 4) 2 ^ {n - 1}} - z ^ {(2m + 1) 2 ^ {n - 1}}\right)
			+ 
			z ^ {3m + 2} - z ^ {3m + 2} \\
			& = 
			\mu h _ m(\mu, z)
		\end{align*}
		and
		\begin{align*}
			\mathcal T h _ 0(\mu, \cdot)(z)
			=
			\sum _ {n = 0} ^ \infty
				\mu ^ n z ^ {T(2 ^ {n + 2})} 
			=
			\sum _ {n = 1} ^ \infty
				\mu ^ n z ^ {2 ^ {n + 1}}
			+
			z ^ 2
			= 
			\mu h _ 0(\mu, z),
		\end{align*}
		which shows that the functions $h _ m(\mu, \cdot)$, $m \geq 0$, are eigenvectors of $\mathcal T$ associated to the eigenvalue $\mu$ as soon as they belong to $\mathcal X _ \omega$.
	\end{proof}

	Our aim is now to show that if the weight $\omega$ is bounded from below, the eigenvectors $h _ m(\mu, \cdot)$, $m \geq 0$, $\mu \in \mathbf D$, span a dense subspace of $\mathcal X _ \omega$.
	
	\begin{Thm}\label{densityeigenvectors}
		If 
		$
			\omega
		$
		is bounded from below,
		then 
		$
			\text{span}[
				h _ m(\mu, \cdot);
				m \geq 0, \mu \in \mathbf D
			]
		$ 
		is dense in $\mathcal X _ \omega$.
		
		In particular if $\omega = \omega _ 0$ then 
		$
			\{
				h _ m(\mu, \cdot); 
				m \geq 0, 
				\mu \in \mathbf D
			\}
		$
		spans a dense subspace of $\mathcal X$.
	\end{Thm}
	
	The proof of Theorem \ref{densityeigenvectors} relies on the following lemma:
		
	\begin{Lem}\label{sequences}
		Let 
		$
			f \in \mathcal X _ \omega,
		$
		with
		$
			f(z) =
			\sum _ {k = 3} ^ \infty
				c _ k z ^ k,
		$ 
		such that for every $n \geq 0$ and every $m \geq 1$
		\[
		\frac{
			c _ {(6m + 4)2 ^ n}
		}{
			\omega((6m + 4)2 ^ n)
		}
		=
		\frac{
			c _ {(2m + 1)2 ^ n}
		}{
			\omega((2m + 1)2 ^ n)
		}
		\quad
		\text{ and }
		\quad
		c _ {2 ^ {n + 2}} = 0.
		\]
		Then for every $k \geq 3$ such that $k \notin \{2 ^ i; i \geq 2\}$, there exist integer sequences $(m _ n) _ {n \geq 1}$, $(p _ n) _ {n \geq 1}$ and $(j _ n) _ {n \geq 1}$ depending on $k$ such that:
		\begin{enumerate}[(i)]
			\item \label{sequencesi}
			$
			c _ k =
			(
			\omega((2m _ 1 + 1)2 ^ {p _ 1})
			/
			\omega((6m _ n + 4)2 ^ {p _ n})
			)
			c _ {(6m _ n + 4)2 ^ {p _ n}}
			$
			for every $n \geq 1$;
			\item \label{sequencesii}
			$m _ n \geq 1$ for every $n \geq 1$;
			\item \label{sequencesiii}
			$3m _ n + 2 = T ^ {j _ n}(k)$ for every $n \geq 1$;
			\item \label{sequencesiv}
			$j _ {n + 1} > j _ n$ 
			if and only if 
			$3 m _ n + 2 \notin \{2 ^ i; i \geq 2\}$
			for every $n \geq 1$;
			\item \label{sequencesv}
			$(j _ n) _ {n \geq 1}$ and $((6m _ n + 4)2 ^ {p _ n}) _ {n \geq 1}$ are either both strictly increasing or both stationary.
		\end{enumerate}
	\end{Lem}
	\begin{proof}
		Since $k \geq 3$ and $k \notin \{2 ^ i; i \geq 2\}$, there exist $m _ 1 \geq 1$ and $p _ 1 \geq 0$ such that $k = (2m _ 1 + 1)2 ^ {p _ 1}$.
		Then 
		\[
		c _ k =
		c _ {(2m _ 1 + 1)2 ^ {p _ 1}} =
		\frac{
			\omega((2m _ 1 + 1)2 ^ {p _ 1})
		}{
			\omega((6m _ 1 + 4)2 ^ {p _ 1})
		}
		c _ {(6m _ 1 + 4)2 ^ {p _ 1}}
		\]
		and
		\[
		3 m _ 1 + 2 = 
		T(2m _ 1 + 1) =
		T ^ {p _ 1 + 1}((2m _ 1 + 1)2 ^ {p _ 1}) =
		T ^ {p _ 1 + 1}(k).
		\]
		We set then $j _ 1 = p _ 1 + 1$.
		
		Suppose now that $(m _ i) _ {1 \leq i \leq n}$, $(p _ i) _ {1 \leq i \leq n}$ and $(j _ i) _ {1 \leq i \leq n}$ have been defined so as to satisfy the following properties:
		\begin{enumerate}[(i)]
			\item
				$
					c _ k =
					(\omega((2m _ 1 + 1) 2 ^ {p _ 1}) / \omega((6m _ i + 4)2 ^ {p _ i}))c _ {(6m _ i + 4) 2 ^ {p _ i}}
				$
				for every $1 \leq i \leq n$;
			\item
				$
					m _ i \geq 1
				$
				for every $1 \leq i \leq n$;
			\item
				$
					3m _ i + 2 = T ^ {j _ i}(k)
				$
				for every $1 \leq i \leq n$;
			\item
				$
					j _ {i + 1} > j _ i
				$
				if and only if 
				$
					3m _ i + 2 \notin 
					\{
						2 ^ p; p \geq 2
					\}
				$
				for every $1 \leq i \leq n - 1$.
		\end{enumerate}
		Let us construct $m _ {n + 1}$, $p _ {n + 1}$ and $j _ {n + 1}$.
		
		If $3 m _ n + 2 \in \{2 ^ i ; i \geq 2\}$, then we set $m _ {n + 1} = m _ n$, $p _ {n + 1} = p _ n$ and $j _ {n + 1} = j _ n$.
		Otherwise there exist $m _ {n + 1} \geq 1$ and $q _ {n + 1} \geq 0$ such that 
		$
		3 m _ n + 2 =
		(2 m _ {n + 1} + 1)2 ^ {q _ {n + 1}}.
		$
		So 
		\[
		(6m _ n + 4)2 ^ {p _ n} = 
		(3 m _ n + 2)2 ^ {p _ n + 1} =
		(2 m _ {n + 1} + 1)2 ^ {q _ {n + 1} + p _ n + 1}.
		\]
		We set then $p _ {n + 1} = q _ {n + 1} + p _ n + 1 > p _ n$.
		Moreover 
		\begin{align*}
		c _ k 
		& =
		\frac{
			\omega((2m _ 1 + 1)2 ^ {p _ 1})
		}{
			\omega((2 m _ {n + 1} + 1)2 ^ {p _ {n + 1}})
		}
		c _ {(2m _ {n + 1} + 1)2 ^ {p _ {n + 1}}} \\
		& =
		\frac{
			\omega((2m _ 1 + 1)2 ^ {p _ 1})
		}{
			\omega((2m _ {n + 1} + 1)2 ^ {p _ {n + 1}})
		}
		\frac{
			\omega((2m _ {n + 1} + 1)2 ^ {p _ {n + 1}})
		}{
			\omega((6m _ {n + 1} + 4)2 ^ {p _ {n + 1}})
		}
		c _ {(6m _ {n + 1} + 4)2 ^ {p _ {n + 1}}} \\
		& =
		\frac{
			\omega((2m _ 1 + 1)2 ^ {p _ 1})
		}{
			\omega((6m _ {n + 1} + 4)2 ^ {p _ {n + 1}})
		}
		c _ {(6m _ {n + 1} + 4)2 ^ {p _ {n + 1}}}.
		\end{align*}
		Remark that 
		\[
		3m _ {n + 1} + 2 
		=
		T ^ {p _ {n + 1} + 1}(
		(2m _ {n + 1} + 1)2 ^ {p _ {n + 1}}
		)
		=
		T ^ {p _ {n + 1} + 1}(
		(6m _ n + 4)2 ^ {p _ n}
		)
		=
		T ^ {p _ {n + 1}}(
		(3 m _ n + 2)2 ^ {p _ n}
		).
		\]
		So
		$
		3m _ {n + 1} + 2 =
		T ^ {p _ {n + 1} - p _ n}(3m _ n + 2) =
		T ^ {p _ {n + 1} - p _ n + j _ n}(k)
		$
		and we set $j _ {n + 1} = p _ {n + 1} - p _ n + j _ n > j _ n$.
		We also have
		$
		(6m _ {n + 1} + 4)2 ^ {p _ {n + 1}} >
		(2m _ {n + 1} + 1)2 ^ {p _ {n + 1}} =
		(6m _ n + 4)2 ^ {p _ n}.
		$
		
		Defined this way, the sequences $(m _ n) _ {n \geq 1}$, $(p _ n) _ {n \geq 1}$ and $(j _ n) _ {n \geq 1}$ satisfy the first four properties and it remains to prove the last one.
		On the one hand, if there exists $n _ 0 \geq 1$ such that 
		$	
		3m _ {n _ 0} + 2 
		\in \{2 ^ i; i \geq 2\}
		$, 
		then the sequences $(j _ n) _ {n \geq 1}$ and $((6m _ n + 4)2 ^ {p _ n}) _ {n \geq 1}$ are stationary by construction. 
		On the other hand if 
		$
		3m _ n + 2 
		\notin \{2 ^ i; i \geq 2\}
		$
		for every $n \geq 1$, then the sequences $(j _ n) _ {n \geq 1}$ and $((6m _ n + 4)2 ^ {p _ n}) _ {n \geq 1}$ are stricly increasing by construction, which proves that the fifth property is satisfied.
	\end{proof}

	\begin{proof}[Proof of Theorem \ref{densityeigenvectors}]
		By Proposition \ref{eigenvectors}, $h _ m(\mu, \cdot)$ belongs to $\mathcal X _ \omega$ for every $m \geq 0$ and every $\mu \in \mathbf D$ since
		$
			\omega
		$
		is bounded from below.
		Let 
		$
			f \in 
			\text{span}[
				h _ m(\mu, \cdot) ;
				m \geq 0, \mu \in \mathbf D
			] ^ \perp
		$
		in $\mathcal X _ \omega$
		with
		$
			f(z) =
			\sum _ {k = 3} ^ \infty
				c _ k z ^ k
		$.
		Our aim is to show that $f = 0$.
		For every $\mu \in \mathbf D$ and every $m \geq 1$
		\begin{align*}
		\varphi _ m (\mu) \colon\! =
		\pdtsca{h _ m(\mu, \cdot), f} =
		\sum _ {n = 0} ^ \infty
		\left(
		\frac{
			\adhe{c _ {(6m + 4) 2 ^ n}}
		}{
			\omega((6m + 4) 2 ^ n)
		}
		- \frac{
			\adhe{c _ {(2m + 1)2 ^ n}}
		}{
			\omega((2m + 1)2 ^ n)
		}
		\right)
		\mu ^ n
		= 0,
		\end{align*}
		\begin{align*}
		\varphi _ 0(\mu) \colon \!=
		\pdtsca{h _ 0(\mu, \cdot), f} =
		\sum _ {n = 0} ^ \infty
		\frac{
			\adhe{c _ {2 ^ {n + 2}}}
		}{
			\omega(2 ^ {n + 2})
		} 
		\mu ^ n
		=0.
		\end{align*}
		For every $k \geq 3$ the radius of convergence of the power series 
		$
		\sum _ {n \geq 0}
		(\adhe{c _ {k2 ^ n}} / \omega(k2 ^ n))z ^ n
		$
		is greater than 1.
		Indeed by the Cauchy-Schwarz's inequality 
		\[
			\left(\sum _ {n = 0} ^ \infty
				\abs{
					\frac{
						\adhe{c _ {k2 ^ n}}
					}{
						\omega(k2 ^ n)
					}
					z ^ n
				}
			\right) ^ 2
			\leq
			\sum _ {n = 0} ^ \infty
				\frac{
					\abs{c _ {k2 ^ n}} ^ 2
				}{
					\omega(k2 ^ n)
				}
			\sum _ {n = 0} ^ \infty
				\frac{
					\abs z ^ {2n}
				}{
					\omega(k2 ^ n)
				}
			< + \infty
		\]
		for every $z \in \mathbf D$ because $f$ belongs to $\mathcal X _ \omega$ and $\omega$ is supposed to be bounded from below.
		Thus the functions $\varphi _ m$ and $\varphi _ 0$ are holomorphic and they identically vanish on $\mathbf D$ for every $m \geq 1$.
		So for every $m \geq 1$
		\[
			\frac{
				c _ {(6m + 4) 2 ^ n}
			}{
				\omega((6m + 4)2 ^ n)
			}
			=
			\frac{
				c _ {(2m + 1)2 ^ n}
			}{
				\omega((2m + 1)2 ^ n)
			}
			\quad
			\text{ and }
			\quad
			c _ {2 ^ {n + 2}} = 0
			\quad
			\text{for every $n \geq 0$.}
		\]
		\begin{Clm}\label{claim}
			Let $\omega$ be a weight on $\mathbf Z _ +$ which is bounded from below, and let $f \in \mathcal X _ \omega$ with
			$
				f(z) =
				\sum _ {k = 3} ^ \infty
					c _ k z ^ k
			$.
			Suppose that for every $m \geq 1$
			\[
				\frac{
					c _ {(6m + 4) 2 ^ n}
				}{
					\omega((6m + 4)2 ^ n)
				}
				=
				\frac{
					c _ {(2m + 1)2 ^ n}
				}{
					\omega((2m + 1)2 ^ n)
				}
				\quad
				\text{ and }
				\quad
				c _ {2 ^ {n + 2}} = 0
				\quad
				\text{for every $n \geq 0$,}
			\]
			Then $c _ k = 0$ for every $k \geq 3$.
		\end{Clm}
		\begin{proof}
			Fix $k \geq 3$.
			If $k \in \{2 ^ i; i \geq 2\}$ then $c _ k = 0$.
			Otherwise consider the sequences $(m _ n) _ {n \geq 1}$, $(p _ n) _ {n \geq 1}$ and $(j _ n) _ {n \geq 1}$ associated to $k$ given by Lemma \ref{sequences}.
			
			On the one hand suppose that
			$
			\{T ^ n(k) ; n \geq 0\}
			\cap
			\{2 ^ i ; i \geq 0\} = \emptyset
			$,
			then in particular
			$
			3m _ n + 2 \notin 
			\{2 ^ i ; i \geq 2\}
			$
			for every $n \geq 1$ 
			by (\ref{sequencesiii}).
			So by (\ref{sequencesiv}), $(j _ n) _ {n \geq 1}$ is strictly increasing and then so is $((6m _ n + 4)2 ^ {p _ n}) _ {n \geq 1}$ by (\ref{sequencesv}).
			Therefore by (\ref{sequencesi})
			\begin{align*}
			\|f \| _ \omega ^ 2 
			=
			\sum _ {j = 3} ^ \infty
			\frac{
				\abs{c _ j} ^ 2
			}{
				\omega(j)
			}
			\geq 
			\sum _ {n = 1} ^ \infty
			\frac{
				\abs{
					c _ {(6 m _ n + 4)2 ^ {p _ n}}
				} ^ 2
			}{
				\omega((6m _ n + 4)2 ^ {p _ n})
			}
			=
			\sum _ {n = 1} ^ \infty
			\frac{
				\omega((6m _ n + 4)2 ^ {p _ n})
			}{
				\omega((2m _ 1 + 1)2 ^ {p _ 1}) ^ 2
			}
			\abs{c _ k} ^ 2.
			\end{align*}
			Since $\omega$ is bounded from below, the previous series can converge only in the case where $c _ k = 0$.
			
			On the other hand suppose that 
			$
			\{T ^ n(k) ; n \geq 0\}
			\cap
			\{2 ^ i; i \geq 0\} 
			\ne \emptyset
			$, then $(j _ n) _ {n \geq 1}$ is stationary.
			Indeed suppose that $(j _ n) _ {n \geq 1}$ is not stationary, that is to say strictly increasing by (\ref{sequencesv}). 
			There exists $N \geq 1$ such that 
			$
			T ^ N(k) \in \{2 ^ i ; i \geq 0\}
			$, 
			so 
			$
			T ^ n(k) \in \{2 ^ i ; i \geq 0\}
			$
			for every $n \geq N$ and 
			$
			T ^ {j _ N}(k) \in \{2 ^ i ; i \geq 0\}
			$ 
			since $j _ N \geq N$.
			Then 
			$
			T ^ {j _ N}(k) 
			=
			3m _ N + 2 
			\in 
			\{2 ^ i ; i \geq 2\}
			$ 
			by (\ref{sequencesii}) and (\ref{sequencesiii}), which implies that $j _ {N + 1} = j _ N$. This is impossible, so $(j _ n) _ {n \geq 1}$ is indeed stationary. 
			Thus there exists $N \geq 1$ such that 
			$
			3m _ {N} + 2 \in 
			\{2 ^ i ; i \geq 2\}
			$
			by (\ref{sequencesiv}),
			which implies by (\ref{sequencesi}) that 
			\[
			c _ k = 
			\frac{
				\omega((2m _ 1 + 1)2 ^ {p _ 1})
			}{
				\omega((6m _ N + 4)2 ^ {p _ N})
			}
			c _ {(6m _ N + 4)2 ^ {p _ N}}
			= 0
			\]
			because 
			$
			(6 m _ N + 4)2 ^ {p _ N} =
			(3m _ N + 2)2 ^ {p _ N + 1}
			\in \{2 ^ i ; i \geq 2\}
			$.
			
			We have shown that $c _ k = 0$ for every $k \geq 3$.
		\end{proof}
		Thus Claim \ref{claim} gives that $f = 0$, which concludes the proof of Theorem \ref{densityeigenvectors}.
	\end{proof}

	We conclude this section by proving that the adjoint $\mathcal T ^ *$ of $\mathcal T$ acting on $\mathcal X _ \omega$ does not admit any eigenvalue.
	We will denote by $\sigma _ p(\mathcal T ^ *)$ the set of its eigenvalues.
	
	\begin{Thm}\label{noeigenvalues}
		If $\omega$ is bounded from below then $\sigma _ p(\mathcal T ^ *) = \emptyset$.
		
		In particular if $\omega = \omega _ 0$ then $\sigma _ p(\mathcal T ^ *) = \emptyset$.
	\end{Thm}

	The proof proceeds along the same lines as the proof of Theorem \ref{densityeigenvectors}.
	We first need the following Lemma, whose proof is similar to the proof of Lemma \ref{sequences} and will be omitted here.
	
	\begin{Lem}\label{sequences2}
		Let $\mu \in \mathbf C \backslash \{0\}$ and let
		$
		f \in \mathcal X _ \omega
		$
		with 
		$
			f(z) = 
			\sum _ {k = 3} ^ \infty
				c _ k z ^ k
			\in \mathcal X _ \omega
		$.
		Suppose that for every $m \geq 1$ and every $n\geq 0$
		\[
		\frac{
			\mu ^ {-1}
			c _ {(3m + 2)2 ^ n}
		}{
			\omega((3m + 2)2 ^ n)
		}
		=
		\frac{
			c _ {(2m + 1)2 ^ n}
		}{
			\omega((2m + 1)2^ n)
		}
		\quad
		\text{and}
		\quad
		c _ {2 ^ {n + 2}} = 0.
		\]
		Then for every $k \geq 3$ such that $k \notin \{2 ^ i; i \geq 2\}$, there exist integer sequences $(m _ n) _ {n \geq 1}$, $(p _ n) _ {n \geq 1}$ and $(j _ n) _ {n \geq 1}$ such that:
		\begin{enumerate}[(i)]
			\item \label{sequences2i}
			$
			c _ k =
			(
			\mu ^ {-n}\omega((2 m _ 1 + 1)2 ^ {p _ 1}) /
			\omega((3m _ n + 2)2 ^ {p _ n})
			)
			c _ {(3m _ n + 2)2 ^ {p _ n}}
			$
			for every $n \geq 1$;
			\item \label{sequences2ii}
			$m _ n \geq 1$ for every $n \geq 1$;
			\item \label{sequences2iii}
			$3m _ n + 2 = T ^ {j _ n}(k)$ for every $n \geq 1$;
			\item \label{sequences2iv}
			$j _ {n + 1} > j _ n$ if and only if 
			$3 m _ n + 2 \notin \{2 ^ i; i \geq 2\}$
			for every $n \geq 1$;
			\item \label{sequences2v}
			$(j _ n) _ {n \geq 1}$ and $((3m _  n + 2)2 ^ {p _ n}) _ {n \geq 1}$ are either both strictly increasing or both stationary.
		\end{enumerate}
	\end{Lem}
	\begin{proof}[Proof of Theorem \ref{noeigenvalues}]
		For every 
		$
		f, g
		\in \mathcal X _ \omega
		$
		with
		$
			f(z) = 
			\sum _ {n = 3} ^ \infty
				c _ n(f) z ^ n
		$
		and with
		$
			g(z) =
			\sum _ {n = 3} ^ \infty
				c _ n(g) z ^ n
		$,
		we have
		\begin{align*}
		\pdtsca{
			\mathcal T f
			\mid
			g
		}
		=
		\sum _ {n = 3} ^ \infty
		\frac{
			c _ n(\mathcal Tf) 
			\adhe{c _ n(g)}
		}{
			\omega(n)
		} 
		=
		\sum _ {n = 3} ^ \infty
		\frac{
			\sum _ {T(j) = n}
			c _ j(f)
			\adhe{c _ n(g)}
		}{
			\omega(n)
		} 
		=
		\sum _ {j \in \{k \geq 3; T(k) \geq 3\}}
		\frac{
			c _ j(f)
			\adhe{c _ {T(j)}(g)}
		}{
			\omega(T(j))
		},
		\end{align*}
		so
		\begin{align*}
		\pdtsca{\mathcal T f \mid g}
		=
		\frac{
			c _ 3(f)
			\adhe{c _ 5(g)}
		}{
			\omega(5)
		}
		+
		\sum _ {k = 5} ^ \infty
		\frac{
			c _ k(f)
			\adhe{c _ {T(k)}(g)}
		}{
			\omega(T(k))
		}.
		\end{align*}
		Then $\mathcal T ^ *$ is defined in the following way:
		\[
			\mathcal T ^ * 
			\sum _ {n = 3} ^ \infty
				c _ n z ^ n
			=
			\frac{
				\omega(3)
			}{
				\omega(5)
			}
			c _ 5
			z ^ 3
			+
			\sum _ {k = 5} ^ \infty
				\frac{
					\omega(k)
				}{
					\omega(T(k))
				}
				c _ {T(k)}
				z ^ k
			\quad
			\text{for every }
			f\colon z \mapsto \sum _ {n = 3} ^ \infty
			c _ n z ^ n
			\in \mathcal X _ \omega.
		\]
		Let $\mu \in \mathbf C$, let $f \in \mathcal X _ \omega$ with 
		$
			f(z) =
			\sum _ {n = 3} ^ \infty
				c _ n z ^ n
		$
		and suppose that $\mathcal T ^ * f = \mu f$.
		We have
		\begin{align*}
		\accolade{l l}{
			(\omega(3) / \omega(5))
			c _ 5 =
			\mu c _ 3 \\
			(\omega(k) / \omega(T(k)))
			c _ {T(k)}
			=
			\mu c _ k,
			&
			k \geq 5\\
			0 = \mu c _ 4,
			&
		}
		\end{align*}
		which is equivalent to
		\begin{align*}
		\accolade{l l}{
			(\omega(2m) / \omega(m)) c _ m
			=
			\mu c _ {2m},
			& 
			m \geq 3 \\
			(\omega(2m + 1) / \omega(3m + 2))c _ {3m + 2} 
			= 
			\mu c _ {2m + 1}, 
			& 
			m \geq 1 \\
			0 = \mu c _ 4.
		}	
		\end{align*}
		Firstly if $\mu = 0$ then $c _ m = 0$ for every $m \geq 3$. 
		So $0$ does not belong to $\sigma _ p(\mathcal T ^ *)$.
		
		We suppose now that $\mu \ne 0$.
		Then we have $\mathcal T ^ * f = \mu f$ if and only if
		\begin{align*}
		\accolade{l l}{
			c _ {m2 ^ n} = \mu ^ {-n} (\omega(m2 ^ n) / \omega(m)) c _  m,
			&
			m \geq 3,
			n \geq 0\\
			\mu ^ {-1}c _ {(3m + 2)2 ^ n} / \omega((3m + 2)2 ^ n)
			=
			c _ {(2m + 1)2 ^ n} / \omega((2m + 1)2 ^ n),
			&
			m \geq 1, 
			n \geq 0 \\
			c _ {2 ^ {n + 2}} = 0,
			&
			n \geq 0
		}
		\end{align*}
		because for every $m \geq 1$ and every $n \geq 0$
		\begin{align*}
		c _ {(2m + 1)2 ^ n} 
		= 
		\frac{
			\mu ^ {-1}
			\omega((2m + 1)2 ^ n)
		}{
			\omega((2m + 1)2 ^ {n - 1})
		}
		c _ {(2m + 1)2 ^ {n - 1}} 
		=
		\ldots 
		=
		\frac{
			\mu ^ {-n}
			\omega((2m + 1)2 ^ n)
		}{
			\omega(2m + 1)
		}
		c _ {2m + 1} ,
		\end{align*}
		so 
		\begin{align*}
		c _ {(2m + 1)2 ^ n}
		=
		\frac{
			\mu ^ {-n - 1}
			\omega((2m + 1)2 ^ n)
		}{
			\omega(3m + 2)
		}
		c _ {3m + 2}
		=
		\ldots 
		=
		\frac{
			\mu ^ {-n - 1 + n}
			\omega((2m + 1)2 ^ n)
		}{
			\omega((3m + 2)2 ^ n)
		}
		c _ {(3m + 2)2 ^ n}.		
		\end{align*}
		Fix $k \geq 3$.
		If $k \in \{2 ^ i ; i \geq 2\}$ then $c _ k = 0$.
		Suppose now that 
		$
			k \notin \{2 ^ i; i \geq 2\}
		$.
		If $\abs\mu \leq 1$, we can remark that 
		\[
			\|f \| _ \omega ^ 2 =
			\sum _ {j = 3} ^ \infty
			\frac{
				\abs{c _ j} ^ 2
			}{
				\omega(j)
			}
			\geq
			\sum _ {n = 0} ^ \infty
			\frac{
				\abs{c _ {k2 ^ n}} ^ 2
			}{
				\omega(k2 ^ n)
			}
			=
			\sum _ {n = 0} ^ \infty
			\frac{
				\abs{\mu} ^ {-2n}
				\omega(k2 ^ n)
			}{
				\omega(k) ^ 2
			}
			\abs{c _ k} ^ 2.
		\]
		Since $\omega$ is bounded from below, the previous series can converge only in the case where $c _ k = 0$ and it follows that $f= 0$.
		Thus $\mu$ does not belong to $\sigma _ p(\mathcal T ^ *)$.
		If $\abs\mu > 1$, consider the sequences $(m _ n) _ {n \geq 1}$, $(p _ n) _ {n \geq 1}$ and $(j _ n) _ {n \geq 1}$ associated to $k$ given by Lemma \ref{sequences2}.
		
		On the one hand suppose that
		$	
		\{
		T ^ n(k); n \geq 0
		\} \cap
		\{
		2 ^ i ; i \geq 0
		\}
		= \emptyset
		$,
		then in particular $3m _ n + 2 \notin \{2 ^ i; i \geq 2\}$ for every $n \geq 1$
		by (\ref{sequences2iii}).
		So by (\ref{sequences2iv}), $(j _ n) _ {n \geq 1}$ is strictly increasing and then so is $((3m _ n + 2)2 ^ {p _ n}) _ {n \geq 1}$ by (\ref{sequences2v}).
		Therefore by (\ref{sequences2i})
		\[
		\|f \| _ \omega ^ 2 = 
		\sum _ {j = 3} ^ \infty
		\frac{
			\abs{c _ j} ^ 2
		}{
			\omega(j)
		}
		\geq
		\sum _ {n = 1} ^ \infty
		\frac{
			\abs{c _ {(3m _ n + 2)2 ^ {p _ n}}} ^ 2
		}{
			\omega((3m _ n + 2)2 ^ {p _ n})
		}
		=
		\sum _ {n = 1} ^ \infty
		\frac{
			\abs{\mu} ^ {2n}
			\omega((3m _ n + 2)2 ^ {p _ n})
		}{
			\omega((2 m _ 1 + 1)2 ^ {p _ 1}) ^ 2
		}
		\abs{c _ k} ^ 2.
		\]
		Since $\omega$ is bounded from below, the previous series converges only in the case where $c _ k = 0$.
		
		On the other hand suppose that
		$
		\{T ^ n(k) ; n \geq 0\}
		\cap
		\{2 ^ i; i \geq 0\} 
		\ne \emptyset
		$, 
		we will follow the lines of the proof of Theorem \ref{densityeigenvectors} and claim that $(j _ n) _ {n \geq 1}$ is stationary.
		Indeed suppose that $(j _ n) _ {n \geq 1}$ is not stationary, that is to say is strictly increasing by (\ref{sequences2v}). 
		There exists $N \geq 1$ such that
		$
		T ^ N(k) \in \{2 ^ i ; i \geq 0\}
		$,
		so
		$
		T ^ n(k) \in \{2 ^ i ; i \geq 0\}
		$
		for every $n \geq N$ and
		$
		T ^ {j _ N}(k) \in \{2 ^ i ; i \geq 0\}
		$
		since $j _ N \geq N$.
		Then $T ^ {j _ N}(k) = 3m _ N + 2 \in \{2 ^ i ; i \geq 2\}$ by (\ref{sequences2ii}) and (\ref{sequences2iii}), which implies that $j _ {N + 1} = j _ N$. This is impossible. 
		So $(j _ n) _ {n \geq 1}$ is stationary and there exists $N \geq 1$ such that 
		$
		3m _ {N} + 2 \in 
		\{2 ^ i ; i \geq 2\}
		$ by (\ref{sequences2iv}).
		This implies by (\ref{sequences2i}) that 
		\[
		c _ k = 
		\frac{
			\mu ^ {-N}
			\omega((2m _ 1 + 1)2 ^ {p _ 1})
		}{
			\omega((3m _ N + 2)2 ^ {p _ N})
		}
		c _ {(3m _ N + 2)2 ^ {p _ N}}
		= 0
		\]
		because $(3 m _ N + 2)2 ^ {p _ N} \in \{2 ^ i; i \geq 2\}$.
		We have shown that $c _ k = 0$ for every $k \geq 3$.
		Thus $f = 0$ and so $\mu$ does not belong to $\sigma _ p(\mathcal T ^ *)$ either.
	\end{proof}

\section{Linear dynamics of the operator $\mathcal T$}

	Our aim is now to understand the links between the dynamics of the Collatz map and those of the operator $\mathcal T$.
	We begin by recalling briefly the definition of some important properties in linear dynamics (hypercyclicity, chaos, frequent hypercyclicity and ergodicity), before investigating whether the operator $\mathcal T$ acting on $\mathcal X _ \omega$ satisfies each of these properties.

	\subsection{Hypercyclicity}
	
	A much studied dynamical property is called \textit{hypercyclicity}. We refer the reader to the books \cite{bay09} and \cite{gro11} for an in-depth study of this notion. 
	Let $X$ be a separable infinite-dimensional complex Banach (or Fréchet) space. We will denote by $\mathcal B(X)$ the set of bounded (or continuous) linear operators on $X$.

	\begin{Def}[{\cite[Definition 2.15]{gro11}}]
		An operator $\mathcal A \in \mathcal B(X)$ is said to be \textit{hypercyclic} if there exists $x \in X$ such that its orbit
		$
		\{
		\mathcal A ^ n x ; 
		n \geq 0
		\}
		$ 
		under $\mathcal A$ is dense in $X.$
		In this case $x$ is called a hypercyclic vector for $\mathcal A$.
	\end{Def}

	Recall that a $G _ \delta$-set in $X$ is a countable intersection of open sets of $X$.
	The following theorem shows that either the hypercyclic vectors of an operator $\mathcal A \in \mathcal B(X)$ forms a dense $G _ \delta$-set in $X$ or $\mathcal A$ does not admit any hypercyclic vectors.
	
	\begin{Thm}[Birkhoff's transitivity theorem, {\cite[Theorems 1.16 and 2.19]{gro11}}]
		An operator $\mathcal A \in \mathcal B(X)$ is hypercyclic if and only if for every non-empty open sets $U$ and $V$ in $X$, there exists $n \geq 0$ such that 
		$	
			\mathcal A ^ n(U) \cap V \ne \emptyset.
		$
		In this case, the hypercyclic vectors for $\mathcal A$ form a dense $G _ \delta$-set in $X$.
	\end{Thm}

	Important examples of hypercyclic operators can be found in the literature.
	On the Fréchet space $Hol(\mathbf C)$ of entire functions endowed with the seminorms defined by
	$
		\|f \| _ {n, \infty}
		=
		\sup \{
			\abs{f(z)};
			\abs z \leq n
		\}
	$
	for $n \geq 1$, the Birkhoff's operator $T ^ {(a)} \colon f \mapsto f(\cdot + a)$ is hypercyclic if and only if $a \ne 0$.
	The derivation operator $D$ acting on $Hol(\mathbf C)$ is also hypercyclic.
	In the Banach space setting, the simplest examples of hypercyclic operators are given by the Rolewicz's operators $\lambda B$ for $\abs \lambda > 1$, where $B$ is the backward shift on $\ell ^ 2(\mathbf N)$ defined by 
	$
		B(x _ n) _ {n \geq 1} =
		(x _ {n + 1}) _ {n \geq 1}
	$.\\

	The following criterion provides a pratical mean to prove the hypercyclicity of an operator.

	\begin{Thm}[Hypercyclicity Criterion, {\cite[Theorem 3.12]{gro11}}] 
		\label{HCC}
		Let $\mathcal A \in \mathcal B(X)$.
		Suppose that there exist dense subsets $X _ 0$ and $Y _ 0$ of $X$, an increasing sequence $(n _ k) _ {k \geq 1}$ of integers and a sequence $(S _ {n _ k} \colon Y _ 0 \to X) _ {k \geq 1}$ of maps such that:
		\begin{enumerate}[(i)]
			\item
			$
			\mathcal A ^ {n _ k} x 
			\to 0
			$
			as 
			$
			k \to + \infty
			$ 
			for every $x \in X _ 0$;
			\item
			$
			S _ {n _ k} y 
			\to 0
			$
			as 
			$
			k \to + \infty
			$ 
			for every $y \in Y _ 0$;
			\item
			$
			\mathcal A ^ {n _ k} S _ {n _ k} y
			\to y
			$
			as 
			$
			k \to + \infty
			$ 
			for every $y \in Y _ 0.$
		\end{enumerate}
		Then the operator $\mathcal A$ is hypercyclic.
	\end{Thm}
	
	In the case where $\omega = \omega _ 0$, and provided that the map $T$ does not admit any non-trivial cycle, it is shown is \cite[Theorem 2.2]{nek21} that $\mathcal T$ satisfies the Hypercyclicity Criterion, and is thus hypercyclic.
	Our first theorem generalizes this result by showing that $\mathcal T$ is hypercyclic on $\mathcal X _ \omega$ under a rather mild hypothesis on the weight $\omega$, independently of any assumption on the existence of cycles for the Collatz map.
	
	\begin{Thm}\label{hypercyclicity}
		If 
		$
			\omega
		$ 
		is bounded from below
		and if
		$
			\omega(k2 ^ n) \to + \infty
		$
		as 
		$
			n \to + \infty
		$
		for every $k \geq 3$,
		then $\mathcal T$ is hypercyclic.
		
		In particular if $\omega = \omega _ 0$ then $\mathcal T$ is hypercyclic on $\mathcal X$.
	\end{Thm}
	\begin{proof}
		It suffices to prove that $\mathcal T$ satisfies the Hypercyclicity Criterion (Theorem \ref{HCC}). 
		Consider the subspace
		$
			X _ 0 = 
			\text{span}[
				h _ m(\mu, \cdot) ; 
				\mu \in \mathbf D, m \geq 0
			]
		$,
		which is dense in $\mathcal X _ \omega$ by Theorem \ref{densityeigenvectors}.
		We have 
		$
			\mathcal T ^ n h _ m(\mu, \cdot) =
			\mu ^ n h _ m(\mu, \cdot)
			\to 
			0
		$ 
		as $n \to + \infty$ for every $\mu \in \mathbf D$ and every $m \geq 0$.
		Thus this is true for any linear combination in $X _ 0$.
		Consider now the dense subspace
		$
			Y _ 0 = 
			\text{span}[
				z ^ k ;
				k \geq 3
			]
		$
		in $\mathcal X _ \omega$ and the map $S \colon Y _ 0 \to Y _ 0$ defined by 
		$
			Sz ^ k = z ^ {2k}
		$
		for every $k \geq 3$.
		Then
		$
			\|S ^ n z ^ k\| _ \omega ^ 2 =
			\|z ^ {k2 ^ n} \| _ \omega ^ 2 =
			1 / \omega(k2 ^ n)
			\to 0
		$
		as $n \to + \infty$ and 
		$
			\mathcal T S z ^ k = 
			\mathcal T z ^ {2k} = 
			z ^ {T(2k)} = 
			z ^ k
		$
		for every $k \geq 3$.
		Thus this is still true for any linear combination in $Y _ 0$.
		Then $\mathcal T$ satisfies the Hypercyclicity Criterion, so $\mathcal T$ is hypercyclic.
	\end{proof}

	Actually, the role played by eigenvectors in the hypercyclicity of an operator appears explicitly in the following criterion.
	
	\begin{Thm}[Godefroy-Shapiro Criterion, {\cite[Theorem 3.1]{gro11}}]\label{godefroy-shapiro}
		Let $\mathcal A \in \mathcal B(X)$.
		Suppose that the subspaces
		\[
			X _ 0 = 
			\text{span}[
				\ker(\mathcal A - \mu) ;
				\abs \mu < 1
			]
			\quad
			\text{and}
			\quad
			Y _ 0 =
			\text{span}[
				\ker(\mathcal A - \mu) ;
				\abs\mu > 1
			]
		\]
		are dense in $X$.
		Then $\mathcal A$ is hypercyclic.
	\end{Thm}

	This criteria is satisfied as soon as there are enough complex numbers $\mu$ such that $h _ m(\mu, \cdot)$ belongs to $\mathcal X _ \omega$ for every $m \geq 0$.
	
	\begin{Thm}
		If $\omega$ is bounded from below and if there exists $\rho > 1$ such that for every $k \geq 3$ the sequence $(\rho ^ n / \omega(k2 ^ n)) _ {n \geq 0}$ is bounded, then $\mathcal T$ satisfies the Godefroy-Shapiro Criterion.
		
		In particular if $\omega = \omega _ 0$, then taking $\rho = 2$ gives us that $\mathcal T$ satisfies the Godefroy-Shapiro Criterion.
	\end{Thm}
	\begin{proof}
		By Proposition \ref{eigenvectors}, $h _ m(\mu, \cdot)$ belongs to $\mathcal X _ \omega$ for every $m \geq 0$ and every $\mu \in \mathbf C$ such that $\abs\mu < \sqrt \rho$.
		By Theorem \ref{densityeigenvectors}, it suffices to show that the subspace
		$
			Y _ 0 =
			\text{span}[
				h _ m(\mu, \cdot) ;
				m \geq 0, 
				1 < \abs \mu < \sqrt \rho
			]
		$
		is dense in $\mathcal X _ \omega$.
		This can be proved by following the steps of the proof of Theorem \ref{densityeigenvectors} and considering the holomorphic functions $\varphi _ m$ and $\varphi _ 0$ on 
		$
			\{
				\mu \in \mathbf C ;
				\abs \mu < \sqrt \rho
			\}
		$
		instead of $\mathbf D$.
	\end{proof}

	\subsection{Frequent hypercyclicity and ergodicity}
	
	The notion of frequent hypercyclicity is a reinforcement of that of hypercyclicity.
	It quantifies the frequency with which the orbit of a vector visits a non-empty open set.
	We refer the reader to \cite{bay06} for more on this notion.
	
	\begin{Def}[{\cite[Definition 9.2]{gro11}}]
		An operator $\mathcal A \in \mathcal B(X)$ is said to be \textit{frequently hypercyclic} if there exists a vector $x \in X$ such that for every non-empty open set $U$ in $X$
		\[
			\liminf _ {N \to + \infty}
				\frac{
					\text{card}\{
						0 \leq n \leq N ; 
						\mathcal A ^ n x \in U
					\}
				}{
					N + 1
				}
			> 0.
		\]
		In this case $x$ is called a frequently hypercyclic vector for $\mathcal A$.
	\end{Def}

	The Hypercyclicity Criterion admits a frequently hypercyclic version, which is the following theorem. 

	\begin{Thm}[Frequent Hypercyclicity Criterion, {\cite[Theorem 9.9]{gro11}}]
		Let $\mathcal A \in \mathcal B(X)$.
		Suppose that there exist a dense set $X _ 0$ and a map $S \colon X _ 0 \to X _ 0$ such that for every $x \in X _ 0$:
		\begin{enumerate}
			\item 
				$
					\sum _ {n \geq 0}
						\mathcal A ^ n x
				$
				converges unconditionally;
			\item
				$
					\sum _ {n \geq 0}
						S ^ n x
				$
				converges unconditionally;
			\item
				$\mathcal A S x = x$.
		\end{enumerate}
		Then $\mathcal A$ is frequently hypercyclic.
	\end{Thm}

	The three hypercyclic operators presented before are actually frequently hypercyclic.
	The operators $\lambda B$, $D$ and $T ^ {(a)}$ are frequently hypercyclic respectively on $\ell ^ 2(\mathbf N)$ and $H(\mathbf C)$ if $\abs\lambda > 1$ and if $a \ne 0$.

	In order to prove that $\mathcal T$ is frequently hypercyclic, we will actually show that $\mathcal T$ is ergodic with respect to a Gaussian measure of full support.

	Let $H$ be a complex Hilbert space, let $\mathcal B$ be the $\sigma$-algebra of Borel subsets of $H$ and let $m$ be a probability measure on $(H, \mathcal B)$.
	
	\begin{Def}[{\cite[Definition 3.9]{bay06}}]
		A transformation
		$
			\mathcal A \colon (H, \mathcal B, m) \to (H, \mathcal B, m)
		$
		in $\mathcal B(H)$ is said to \textit{preserve the measure} $m$ if 
		$
			m(\mathcal A ^ {-1}(B)) = m(B)
		$
		for every $B \in \mathcal B$.
		A measure-preserving transformation $\mathcal A \colon (H, \mathcal B, m) \to (H, \mathcal B, m)$ is said to be \textit{ergodic} if for every $B \in \mathcal B$,
		$
			\mathcal A ^ {-1}(B) = B
		$ 
		implies that $m(B) \in \{0, 1\}$.
	\end{Def}

	The eigenvectors of an operator play an important role in the study of its dynamics.
	The fact that the eigenvectors of $\mathcal T$ span a dense subspace of $\mathcal X _ \omega$ allowed us to show that $\mathcal T$ is hypercyclic under some assumptions on the weight $\omega$. 
	In order to show that $\mathcal T$ is ergodic with respect to a Gaussian measure with full support, we will rely on the properties of the eigenvectors of $\mathcal T$ associated to unimodular eigenvalues.
	We first recall a few revelant definitions.

	\begin{Def}[{\cite[Definition 3.1]{bay06}}]
		An operator $\mathcal A \in \mathcal B(H)$ is said to have a 
		\textit{perfectly spanning set of eigenvectors associated to unimodular eigenvalues}
		if there exists a continuous probability measure $\sigma$ on the unit circle $\mathbf T$ such that the eigenvectors of $\mathcal A$ associated to eigenvalues $\mu \in A$ span a dense subspace in $H$ for every $\sigma$-measurable subset $A$ of $\mathbf T$ satistying $\sigma(A) = 1$.
	\end{Def}

	\begin{Def}[{\cite[Definitions 3.13 and 3.14]{bay06}}]
		Let $(\Omega, \mathcal F, P)$ be a probability space.
		A measurable function 
		$
			f \colon (\Omega, \mathcal F, P) \to \mathbf C
		$ 
		is said to have 
		\textit{complex symmetric Gaussian distribution}
		if $\Re(f)$ and $\Im(f)$ have independent centered Gaussian distribution and the same variance. 
		
		A probability measure $m$ on $(H, \mathcal B)$ is said to be a \textit{Gaussian measure} if the function $y \mapsto \pdtsca{y, x}$ has complex symmetric Gaussian distribution for every $x \in H$.
		
		A measure $m$ on $(H, \mathcal B)$ is said to have a \textit{full support} if $m(U) > 0$ for every open set $U \in \mathcal B$.
	\end{Def}

	\begin{Thm}[{\cite[Theorem 3.22]{bay06}}]
		\label{perfectly spanning}
		If an operator $\mathcal A \in \mathcal B(H)$ admits a perfectly spanning set of eigenvectors associated to unimodular eigenvalues, then there exists a Gaussian invariant measure with full support $m$ on $H$ such that 
		$
			\mathcal A \colon (H, \mathcal B, m) \to (H, \mathcal B, m)
		$ 
		is ergodic.
	\end{Thm}

	We are now ready to prove the ergodicity of $\mathcal T$.

	\begin{Thm}\label{ergodicity}
		If 
		$
			\omega
		$
		is bounded from below
		and if 
		$
			\sum _ {n = 0} ^ \infty
				1 / \omega(k2 ^ n)
			< + \infty
		$
		for every $k \geq 3$, then $\mathcal T$ is ergodic with respect to a Gaussian invariant measure with full support on $\mathcal X _ \omega$.
		
		This is in particular the case when $\omega = \omega _ 0$.
	\end{Thm}
	\begin{proof}
		By Theorem \ref{perfectly spanning}, it suffices to prove that $\mathcal T$ admits a perfectly spanning set of unimodular eigenvectors. 
		Since 
		$
			\sum _ {n = 0} ^ \infty
				1 / \omega(k2 ^ n)
			< + \infty
		$
		for every $k \geq 3$, 
		the function $h _ m(\mu, \cdot)$ belongs to $\mathcal X _ \omega$ for every $m \geq 0$ and every $\mu \in \mathbf C$ such that $\abs\mu \leq 1$ by Proposition \ref{eigenvectors}.
		We will prove that for every Borel set $A \subset \mathbf T$ such that $\sigma(A) = 1$, where $\sigma$ is the Lebesgue measure on the unit circle, the subspace
		$
			\text{span}[
				h _ m(\mu, \cdot) ;
				m \geq 0, \mu \in A
			]
		$ is dense in $\mathcal X _ \omega$.
		\begin{Clm}\label{claim2}
			Let $\omega$ be a weight on $\mathbf Z _ +$ which is bounded from below and such that
			$
				\sum _ {n = 0} ^ \infty
					1 / \omega(k2 ^ n)
				< + \infty
			$			
			for every $k \geq 3$.
			Then the subspace 
			$
				\text{span}[
					h _ m(\mu, \cdot) ;
					m \geq 0, 
					\mu \in D
				]
			$
			is dense in $\mathcal X _ \omega$ for every dense subset $D$ of $\mathbf T$.
		\end{Clm}
		\begin{proof}
			Let 
			$
				f \in 
				\text{span}[
				h _ m(\mu, \cdot);
				m \geq 0,
				\mu \in D
				] ^ \perp
			$
			in $\mathcal X _ \omega$
			with
			$
				f(z) =
				\sum _ {k = 3} ^ \infty
					c _ k z ^ k
			$.
			Our aim is to prove that $f = 0$.
			For every $\mu \in D$ and every $m \geq 1$
			\[
				\varphi _ m(\mu) \colon \!=
				\pdtsca{h _ m(\mu, \cdot), f} =
				\sum _ {n = 0} ^ \infty
					\left(
						\frac{
							\adhe{c _ {(6m + 4)2 ^ n}}
						}{
							\omega((6m + 4)2 ^ n)
						}
						-
						\frac{
							\adhe{c _ {(2m + 1)2 ^ n}}
						}{
							\omega((2m + 1)2 ^ n)
						}
					\right)
					\mu ^ n
				= 0
			\]
			and
			\[
				\varphi _ 0(\mu) \colon \! =
				\pdtsca{h _ 0(\mu, \cdot), f} =
				\sum _ {n = 0} ^ \infty
					\frac{
						\adhe{c _ {2 ^ {n + 2}}}
					}{
						\omega(2 ^ {n + 2})
					}
					\mu ^ n
				= 0.
			\]
			For every $k \geq 3$ the power series 
			$
				\sum _ {n = 0} ^ \infty
					(\adhe{c _ {k2 ^ n}} / \omega(k2 ^ n)) z ^ n
			$
			is uniformly convergent on the closed disk $\adhe{\mathbf D}$.
			Indeed for every $n \geq 0$ and every $z \in \adhe{\mathbf D}$
			\[
				\abs{
					\frac{
						\adhe{c _ {k2 ^ n}}
					}{
						\omega(k2 ^ n)
					}
					z ^ n
				}
				\leq
				\frac{
					\abs{c _ {k2 ^ n}}
				}{
					\sqrt{\omega(k2 ^ n)}
				}
				\frac 1 {\sqrt{\omega(k2 ^ n)}}
			\]
			and by the Cauchy-Schwarz's inequality
			\[
				\left(\sum _ {n = 0} ^ \infty
					\frac{
						\abs{c _ {k2 ^ n}}
					}{
						\sqrt{\omega(k2 ^ n)}
					}
					\frac 1 {\sqrt{\omega(k2 ^ n)}}
				\right) ^ 2
				\leq
				\sum _ {n = 0} ^ \infty
					\frac{
						\abs{c _ {k2 ^ n}} ^ 2
					}{
						\omega(k2 ^ n)
					}
				\sum _ {n = 0} ^ \infty
					\frac 1 {\omega(k2 ^ n)}
				< + \infty
			\]
			since $f$ belongs to $\mathcal X _ \omega$ and by assumption on $\omega$.
			Then $\varphi _ m$ and $\varphi _ 0$ are holomorphic on $\mathbf D$, continuous on $\adhe{\mathbf D}$ and they vanish on $D$ for every $m \geq 1$.
			Since $D$ is dense in the unit circle $\mathbf T$, the functions $\varphi _ m$ and $\varphi _ 0$ vanish on $\mathbf T$ by continuity.
			By the maximum modulus principle, $\varphi _ m$ and $\varphi _ 0$ identically vanish on $\mathbf D$ for every $m \geq 1$ and
			\[
				\frac{
					c _ {(6m + 4)2 ^ n}
				}{
					\omega((6m + 4)2 ^ n)
				}
				=
				\frac{
					c _ {(2m + 1)2 ^ n}
				}{
					\omega((2m + 1)2 ^ n)
				}
				\quad
				\text{and}
				\quad
					c _ {2 ^ {n + 2}}
				= 0
				\quad
				\text{for every $n \geq 0$.}
			\]
			It follows then from Claim \ref{claim} that $f = 0$, which gives that 
			$
				\text{span}[
					h _ m(\mu, \cdot) ;
					m \geq 0, 
					\mu \in D
				]
			$
			is dense in $\mathcal X _ \omega$.
		\end{proof}
		Thus since a Borel set $A \subset \mathbf T$ satisfying $\sigma(A) = 1$ is dense in $\mathbf T$, Claim \ref{claim2} gives that $\mathcal T$ admits a perfectly spanning set of unimodular eigenvectors, which concludes the proof.
	\end{proof}

	We now will link the ergodicity of an operator to its frequent hypercyclicity.
	Let $\mathcal A \in \mathcal B(H)$ be ergodic with respect to an invariant measure with full support.
	First it follows from the Birkhoff's transitivity theorem that $\mathcal A$ is hypercyclic.
	Then it follows from Birkhoff's pointwise ergodic theorem that:
	
	\begin{Thm}[{\cite[Proposition 6.23]{bay09}}]\label{ergodicfrequent}
		If an operator $\mathcal A \in \mathbf B(H)$ is ergodic with respect to a probability measure $m$ on $H$ with full support, then $\mathcal A$ is frequently hypercyclic.
		Moreover the frequently hypercyclic vectors of $\mathcal A$ form a set of full measure for $m$.
	\end{Thm}

	We deduce from Theorems \ref{ergodicity} and \ref{ergodicfrequent} the following result:
	
	\begin{Thm}\label{frequenthypercyclicity}
		If 
		$
			\omega
		$
		is bounded from below
		and if 
		$
		\sum _ {n = 0} ^ \infty
		1 / \omega(k2 ^ n)
		< + \infty
		$
		for every $k \geq 3$, then $\mathcal T$ is frequently hypercyclic on $\mathcal X _ \omega$.
		
		In particular if $\omega = \omega _ 0$, $\mathcal T$ is frequently hypercyclic on $\mathcal X$.
	\end{Thm}

	We finish this section by proving that, under the assumptions of Theorem \ref{ergodicity}, $\mathcal T$ is chaotic.
	
	\begin{Def}[{\cite[Definition 2.29]{gro11}}]
		An operator $\mathcal A \in \mathcal B(X)$ is said to be \textit{chaotic} if $\mathcal A$ is hypercyclic and has a dense set of periodic points.
	\end{Def}
	
	\begin{Thm}\label{chaos}
		If 
		$
		\omega
		$
		if bounded from below
		and if 
		$
		\sum _ {n = 0} ^ \infty
		1 / \omega(k2 ^ n)
		< + \infty
		$
		for every $k \geq 3$, then $\mathcal T$ is chaotic on $\mathcal X _ \omega$.
		
		In particular if $\omega = \omega _ 0$, then $\mathcal T$ is chaotic on $\mathcal X$.
	\end{Thm}
	\begin{proof}
		Since 
		$
			\sum _ {n = 0} ^ \infty
				1 / \omega(k2 ^ n)
			< + \infty
		$
		for every $k \geq 3$, the function $h _ m(\mu, \cdot)$ belongs to $\mathcal X _ \omega$ for every $m \geq 0$ and every $\mu \in \mathbf C$ such that $\abs \mu \leq 1$ by Proposition \ref{eigenvectors}.
		By Theorem \ref{hypercyclicity}, $\mathcal T$ is hypercyclic and it suffices to show that set of its perdiodic points Per($\mathcal T$) is dense in $\mathcal X _ \omega$.
		It follows from \cite[Proposition 2.33]{gro11} that
		$
			\text{Per}(\mathcal T) =
			\text{span}[
				h _ m(\mu, \cdot) ; 
				m \geq 0, 
				\mu \in \{
					e ^ {\alpha i \pi}; \alpha \in \mathbf Q
				\}
			],
		$
		and its density is given by Claim \ref{claim2}.
	\end{proof}
	
	\section{Open questions}
		We finish this paper by presenting a few open questions connected to the results we have presented.
		\begin{Ques}
			In the case where $\omega = \omega _ 0$, it follows from Propositions \ref{norm} and \ref{eigenvectors} that 
			\[
				\{
					z \in \mathbf C ;
					\abs z < \sqrt 2
				\}
				\subset
				\sigma _ p(\mathcal T)
				\subset 
				\sigma(\mathcal T)
				\subset 
				\{
					z \in \mathbf C ;
					\abs z \leq 8 / 3
				\}.
			\]
			What is the exact value of the spectral radius $\rho(\mathcal T)$ of $\mathcal T$ ?
			Is it true that $\rho(\mathcal T) = \sqrt 2$ ?
			
			In order to answer this question, one could perhaps build on the following proposition, which provides an explicit expression for the norm of the iterates $\mathcal T ^ n$ of $\mathcal T$:
			\begin{Prop}
				If $\omega = \omega _ 0$, then for every $n \geq 1$
				\[
				\| \mathcal T ^ n\| ^ 2 =
				\max _ {0 \leq r \leq 3 ^ n - 1}
				\sum _ {a\xi + b \in P ^ {(n)} _ {n, r} \subset \mathbf Q [\xi]}
				\frac a {3 ^ n}
				\]
				where 
				$
				P ^ {(n)} _ {0, r} = \{3 ^ n\xi + r\}
				$
				and
				$
				P^ {(n)} _ {k + 1, r} =
				\{
				2P ; P \in P ^ {(n)} _ {k, r}
				\}
				\cup 
				\{
				(2P - 1) / 3 ;
				P \in P ^ {(n)} _ {k, r}, 
				P(0) \equiv 2 \mod 3
				\}
				$ 
				for every $k \geq 0$ and $0 \leq r \leq 3 ^ n - 1$.
			\end{Prop}
			\begin{proof}
				According to Proposition \ref{norm}
				\[
				\|
				\mathcal T ^ n 
				\| ^ 2 =
				\max _ {0 \leq r \leq 3 ^ n - 1}
				\left\{
				\sup _ {m \geq \delta _ r}
				\sum _ {T ^ n(j) = 3 ^ nm + r}
				\frac{
					j + 1
				}{
					3 ^ n m + r + 1
				}
				\right\}
				\]
				where $\delta _ r = 1$ if $r \leq 2$ and $\delta _ r = 0$ otherwise.
				Since $T(j) = k$ if and only if $j = 2k$ or $j = (2k - 1) / 3$ if $k \equiv 2 \mod 3$, the sets $P ^ {(n)} _ {n, r}$ are such that 
				$
				T ^ n(j) = 3 ^ n m + r
				$ 
				if and only if 
				$
				j \in \{P(m) ; P \in P ^ {(n)} _ {n, r}\}
				$
				for every $0 \leq r \leq 3 ^ n - 1$ and every $m \geq \delta _ r$.
				Then
				\begin{align*}
				\|
				\mathcal T ^ n 
				\| ^ 2 
				& =
				\max _ {0 \leq r \leq 3 ^ n - 1}
				\left\{
				\sup _ {m \geq \delta _ r}
				\sum _ {a\xi + b \in P ^ {(n)} _ {n, r}}
				\frac{
					am + b + 1
				}{
					3 ^ n m + r + 1
				}
				\right\}.
				\end{align*}
				Fix $0 \leq r \leq 3 ^ n - 1$.
				If $\alpha, \beta, \gamma, \delta > 0$, one can remark that a sequence $((\alpha m + \beta) / (\gamma m + \delta)) _ {m \geq 0}$ is strictly increasing or constant if and only if $\alpha \delta \geq \beta \gamma$.
				We will prove by recursion that $a(r + 1) \geq 3 ^ n (b + 1)$ for every $a\xi + b \in P ^ {(n)} _ {n, r}$.
				Indeed this is firstly true for $P ^ {(n)} _ {0, r} = \{3 ^ n \xi + r\}$. 
				Besides if $P ^ {(n)} _ {k, r}$ satisfies it, let $P\in P ^ {(n)} _ {k + 1, r}$.
				Either $P = 2a\xi + 2b$ with $a\xi + b \in P ^ {(n)} _ {k, r}$, or $P = 2a\xi / 3 + (2b - 1) / 3$ with $a\xi + b \in P ^ {(n)} _ {k, r}$ and $b \equiv 2 \mod 3$.
				On the one hand we would have 
				$
				2a(r + 1) \geq 3 ^ n(2b + 2) > 3 ^ n(2b + 1)
				$, 
				and on the other hand we would have 
				$
				2a(r + 1) / 3 \geq 3 ^ n(2b + 2) / 3 \geq 3 ^ n((2b - 1) / 3 + 1)
				$.
				Then this is true for $P ^ {(n)} _ {k + 1, r}$, which proves that $P ^ {(n)} _ {n, r}$ satisfies it.
				So for every $a\xi + b \in P ^ {(n)} _ {n, r}$
				\[
				\sup _ {m \geq \delta _ r}
				\frac{
					am + b + 1
				}{
					3 ^ nm + r + 1
				}
				=
				\lim _ {m \to + \infty}
				\frac{
					am + b + 1
				}{
					3 ^ nm + r + 1
				}
				=
				\frac a {3 ^ n},
				\]
				which concludes the proof.
			\end{proof}
		\end{Ques}

		\begin{Ques}
			If
			$
				\liminf _ {n \to + \infty}
					\omega(n)
				= 0,
			$
			could $\mathcal T$ be bounded, hypercyclic or ergodic ?
			Could $\mathcal T ^ *$ have any eigenvalue ?
		\end{Ques}	
	
		The operator $\mathcal T$ can actually be seen as acting on a space of sequences $(c _ n) _ {n \geq 3}$ instead of on a space of holomorphic functions 
		$
			f \colon z \mapsto
			\sum _ {n = 3} ^ \infty
				c _ n z ^ n
		$.
		\begin{Ques}
			What would remain from these results if we considered $\mathcal T$ as acting on the space of complex sequences $\mathbf C ^ \mathbf N$ ?
		\end{Ques}


\begin{bibdiv}
		\begin{biblist}
			\bib{app02}{article}{
				AUTHOR = {D. Applegate and J. C. Lagarias},
				TITLE = {Lower bounds for the total stopping time of $3x + 1$ iterates},
				JOURNAL = {Math. of Comp.},
				FJOURNAL = {Mathematics of Computation},
				VOLUME = {72},
				YEAR = {2002},
				PAGES = {1035 - 1049},
			}
		
			\bib{bay06}{article}{
				AUTHOR = {F. Bayart and S. Grivaux}
				TITLE = {Frequently hypercyclic operators}
				JOURNAL = {Amer. Math. Soc.}
				FJOURNAL = {American Mathematical Society}
				VOLUME = {358},
				YEAR = {2006},
				PAGES = {5083 - 5117},
			}
		
			\bib{bay09}{book}{
				AUTHOR = {F. Bayart and \'E. Matheron},
				TITLE = {Dynamics of linear operators},
				SERIES = {Cambridge Tracts in Mathematics},
				VOLUME = {179}
				PUBLISHER = {Cambridge University Press},
				YEAR = {2009},
			}
		
			\bib{bay11}{article}{
				AUTHOR = {F. Bayart and \'E. Matheron},
				TITLE = {Mixing operators and small subsets of the circle},
				JOURNAL = {preprint available at https://arxiv.org/abs/1112.1289v1},
				YEAR = {2011},
			}
		
			\bib{ber94}{article}{
				AUTHOR = {L. Berg and G. Meinardus},
				TITLE = {Functional equations connected with the Collatz problem},
				JOURNAL = {Res. in Math.},
				FJOURNAL = {Results in Mathematics},
				VOLUME = {25},
				YEAR = {1994},
				PAGES = {1 - 12},
			}
		
			\bib{eli93}{article}{
				AUTHOR = {S. Eliahou},
				TITLE = {The $3x + 1$ problem: new lower bounds on nontrivial cycle lengths},
				JOURNAL = {Disc. Math.},
				FJOURNAL = {Discrete Mathematics},
				VOLUME = {118},
				YEAR = {1993},
				PAGES = {45 - 56},
			}
		
			\bib{gro11}{book}{
				AUTHOR = {K.-G. Grosse-Erdmann and A. Peris Manguillot},
				TITLE = {Linear chaos},
				SERIES = {Universitext},
				PUBLISHER = {Springer London},
				YEAR = {2011},
			}

			\bib{kra02}{article}{
				AUTHOR = {I. Krasikov and J. C. Lagarias},
				TITLE = {Bounds for the $3x + 1$ problem using difference inequalities},
				JOURNAL = {Acta Arit.},
				FJOURNAL = {Acta Arithmetica},
				VOLUME = {109},
				YEAR = {2002},
				PAGES = {237 - 258},
			}
		
			\bib{lag10}{article}{
				AUTHOR = {J. C. Lagarias},
				TITLE = {The $3x + 1$ problem: an overview},
				JOURNAL = {The Ultimate Challenge: The $3x + 1$ Problem. Edited by J. C. Lagarias, Amer. Math. Soc., Providence, RI},
				YEAR = {2010},
				PAGES = {3 - 29},
			}
		
			\bib{mur13}{article}{
				AUTHOR = {M. Murillo-Arcila and A. Peris},
				TITLE = {Strong mixing measures for linear operators and frequent hypercyclicity},
				JOURNAL = {J. Math. Anal. Appl.},
				FJOURNAL = {Journal of Mathematical Analysis and Applications},
				VOLUME = {398},
				YEAR = {2013},
				PAGES = {462 - 465},
			}

			\bib{nek21}{article}{
				AUTHOR = {M. Neklyudov},
				TITLE = {Functional analysis approach to the Collatz conjecture},
				JOURNAL = {preprint available at https://arxiv.org/abs/2106.11859v9},
				YEAR = {2021},
			}
		
			\bib{oli10}{article}{
				AUTHOR = {T. Oliveira e Silva},
				TITLE = {Empirical verification of the $3x + 1$ and related conjectures},
				JOURNAL = {The Ultimate Challenge: The $3x + 1$ Problem. Edited by J. C. Lagarias, Amer. Math. Soc., Providence, RI},
				YEAR = {2010},
				PAGES = {189-207}				
			}
		
			\bib{tao22}{article}{
				AUTHOR = {T. Tao},
				TITLE = {Almost all orbits of the Collatz map attain almost bounded values},
				JOURNAL = {Forum of Math., Pi}
				FJOURNAL = {Forum of Mathematics, Pi},
				YEAR = {2022},
			}
		\end{biblist}
	\end{bibdiv}
\end{document}